\documentclass[amsart,11pt]{article}
\usepackage{setspace} %REMOVE SETSPACE UPON APPROVAL
\usepackage{amsmath, amsthm, amssymb}
\usepackage{graphicx}
\usepackage{subfigure}
\usepackage{epsfig}
\usepackage{tikz}
\usepackage{hyperref}
\usepackage{chngcntr}
\newtheoremstyle{dotless}{}{}{\itshape}{}{\bfseries}{}{ }{}
\def\Real{\hbox{I\kern-.1667em\hbox{R}}}
\theoremstyle{dotless}	
\input epsf

\counterwithin{figure}{section}

\title{Sign games on graphs}

\author{Liz Blum, Ranjan Rohatgi\\
    Saint Mary's College
    Notre Dame, Indiana  46556\\
    \\
    \\
    Lily Brustkern, Rosetta Hawkins, Neil R. Nicholson\\
    University of Notre Dame\\
    Notre Dame, Indiana  46556 \\
    \\
	}

\newtheoremstyle{dotless}{}{}{\itshape}{}{\bfseries}{}{ }{}
\def\Real{\hbox{I\kern-.1667em\hbox{R}}}
\theoremstyle{dotless}	
\input epsf

\newtheorem{thm}{Theorem}[section]

\newtheorem{lem}[thm]{Lemma}

\begin{document}

\maketitle
%\doublespacing %REMOVE DOUBLESPACE UPON APPROVAL

\begin{abstract}
We define the Sign Game as a two-player game played on a simple undirected mathematical graph $G$. The players alternate turns, assigning vertices of $G$ either $1$ or $-1$, and edges take on the value of the product of their endvertices. The game ends when all vertices are assigned values, and the score of the game is the sum of all edge values. One player's goal is to make the score positive while the other's is to make the score negative. In this paper we investigate the game being played on various types of graphs, determining outcomes and winning strategies. \\

%Secondary 05C57

\end{abstract}

\section{Introduction and definitions}

Nowadays people have strong opinions that oftentimes make their way into online social networks. When two people share the same (or opposite) opinion on a certain touchy subject, we might say that the social network connection between them is \textit{positive} (resp. \textit{negative}) on this matter. Take politics, for example. When two friends in a social network share the same political affiliation, it is natural to say the political connection between them is \textit{positive}, whereas it is \textit{negative} if they have opposite affiliations. 

In this paper, rather than looking at these individual connections, we want to look at the entire friendship network and investigate it as a whole. Social networks are naturally represented modeled with mathematical graphs. Some, like X or Strava, give rise to directed graphs (an edge from $v_1$ to $v_2$ means $v_1$ is \textit{following} $v_2$ in the network; a fan may follow a celebrity but the celebrity likely won't follow the fan), while others, like Facebook, being \textit{friends} is a two-way street (if $v_1$ is friends with $v_2$, then $v_2$ is friends with $v_1$). Undirected simple graphs (those with no loops and no multiple edges) best represent these latter scenarios, and in this paper, all graphs are assumed to be of this type.

A variety of two-player graph-labeling games appear in the literature. In 1991, Bodlaender, in \cite{1} extended the famous map-coloring game to a game on arbitrary finite graphs. Since then, many more games have been introduced and analyzed, including the game chromatic index \cite{3}, several graph labeling games in \cite{6}, the game of cycles \cite{5}, and the graceful game \cite{4}. Here, we define a new two-player game, called the Sign Game, that mimics the concept of measuring how positive or negative the collection of a network's connections are. 

\subsection{Defining the game}

Let $G$ be a simple undirected graph. The \textbf{Sign Game} on $G$ is a two-player game with two players \textbf{Player P} (for \textit{positive}) and \textbf{Player N} (for \textit{negative}) proceeding as follows.

\begin{enumerate}
\item Players alternate turns; the first player to play is referred to as \textbf{Player 1} and the second player as \textbf{Player 2} (so that Player P is one of Player 1 or Player 2, as is Player N). On a turn, a player assigns either a $+1$ or a $-1$ to any vertex on $G$ that has not been assigned a value.
\item When both endvertices of an edge have been assigned a value, that edge takes on the value of the product of the vertices; this is called the \textbf{score} (in points) of the edge. When the second endvertex of an edge is assigned and the edge takes on a value of $+1$ (resp. $-1$), we say that Player P (resp. Player N) \textbf{banks} $1$ (resp. $-1$) point(s). 
\item The game concludes when all vertices have been assigned values (and consequently all edges have a score). Define the \textbf{score} $s(G)$ of the game to be the sum of all edge scores.
\item Player P wins if $s(G) > 0$, Player N wins if $s(G) < 0$, and the game is a draw if $s(G) = 0$. 
\end{enumerate}

The score of the game in Fig. \ref{fig1} is $-1$, resulting in Player N winning. Any game played on the graph in Fig. \ref{fig1} would result in a winner (since the graph has $5$ edges), but would the winner always be Player N? As we will see, the result of a game often, but not always, depends on which player plays first. 

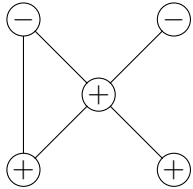
\begin{figure}[htb]
\centering
\begin{tikzpicture} 
%\draw[step=1cm,gray,very thin] (0,0) grid (2,2);
\draw (0,.22)--(0,1.78);
\draw (.15,.15)--(.85,.85);
\draw (.15,1.85)--(.85,1.15);
\draw (1.15,1.15)--(1.85,1.85);
\draw (1.15,.85)--(1.85,.15);

\draw (0,0) circle (.22cm);
\node (00) at (0,0) {$+$};
\draw (0,2) circle (.22cm);
\node (02) at (0,2) {$-$};
\draw (1,1) circle (.22cm);
\node (11) at (1,1) {$+$};
\draw (2,0) circle (.22cm);
\node (20) at (2,0) {$+$};
\draw (2,2) circle (.22cm);
\node (22) at (2,2) {$-$};

\end{tikzpicture}
\caption{Player N wins}
 \label{fig1}
\end{figure}

Throughout this paper we will assume that both players choose to play optimally, with a primary goal of winning and a secondary goal of the game ending in a draw (that is, not allowing their opponent to win). 

\subsection{SG-equivalence}

 Because the nature of the game makes brute force investigation of anything but the smallest graphs overwhelming, we introduce a method to ``reduce" graphs to ones with fewer vertices for which playing the Sign Game on them yields the same result as playing the game on the original graph.

Two graphs $G_1$ and $G_2$, which may have some vertices assigned values, are said to be \textbf{SG-equivalent} if playing the Sign Game to completion on $G_2$ (preserving the order of play for Players P and N) yields an equal score as if it were finished on $G_1$ (that is, $s(G_1) = s(G_2)$). Figure \ref{fig2} illustrates two SG-equivalent graphs.

\begin{figure}[htb]  
\centering  

\subfigure[Graph $G_1$]  
{  
\begin{tikzpicture}[]  

\draw (0.15,0.15)--(2.85,1.85);
\draw (2.1,.19)--(2.9,1.8);
\draw (4,0)--(3.1,1.8);
\draw (6,0)--(3.15,1.85);

\draw (0,0) circle (.22cm);
\node (00) at (0,0) {$+$};
\draw (2,0) circle (.22cm);
\node (20) at (2,0) {$-$};

\draw[black,fill=black] (4,0) circle (.5ex);
\draw[black,fill=black] (6,0) circle (.5ex);

\draw (3,2) circle (.22cm);
\node (32) at (3,2) {$+$};

\end{tikzpicture}  

}  
% The only difference is here, where I have commented out an empty line.
\hspace{50pt}
\subfigure[Graph $G_2$]  
{  

\begin{tikzpicture} 

\draw (1,2) circle (.22cm);
\node (102) at (1,2) {$+$};
%\draw (0,0) circle (.22cm);
%\draw (2,0) circle (.22cm);
\draw[black,fill=black] (0,0) circle (.5ex);
\draw[black,fill=black] (2,0) circle (.5ex);

\draw (0,0)--(.85,1.85);
\draw (1.15,1.85)--(2,0);

\end{tikzpicture}  

}

\caption{SG-equivalent graphs $G_1$ and $G_2$}
\label{fig2}
\end{figure}
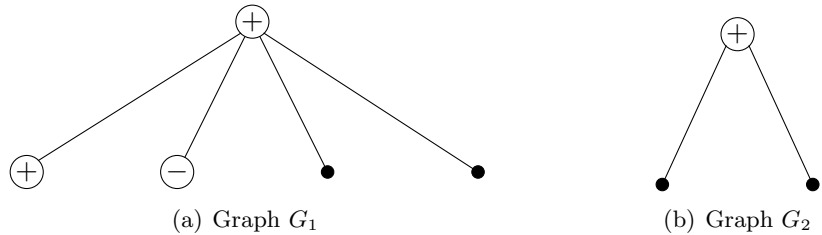

\section{Graph types}

In this section we investigate playing the Sign Game on various categories of mathematical graphs, beginning with one representing a social network where every member is connected to every other member: complete graphs.

\subsection{Complete graphs}

The \textbf{complete graph} $K_n$ ($n \geq 2$) is the graph with $n$ vertices and an edge between every pair of distinct vertices. When the Sign Game is played on $K_n$, Player N almost always wins, as shown in Thm. \ref{thm1}. Before proving it, we state the following lemma, whose proof is a simple counting argument left to the reader. 

\begin{lem} \label{lem1}
If the Sign Game is played to completion on $G = K_n$ and $a$ vertices are assigned $+1$ and $b$ vertices assigned $-1$ (so that $a+b = n$, with $a,b \geq 2$), then 

\begin{equation}
         s(G) = \binom{a}{2}+ \binom{b}{2}-ab.
    \end{equation}   
\end{lem}

%\begin{proof}
%    There are ${a \choose 2}$ (resp. ${b \choose 2}$) edges whose endvertices are both assigned a $+1$ (resp. $-1$) and $ab$ edges whose endvertices are assigned opposite signs. 
%\end{proof}

The result in Lem. \ref{lem1} could be stated solely in terms of $a$ and $n$, but the choice to introduce the variable $b$ allows for a cleaner proof of this section's main result, Thm. \ref{thm1}. In the proof of this theorem, we will see in many instances where a player's strategy is dependent upon what the other player just assigned. While the choice of vertex or vertex assignment will vary, we will refer to such a strategy as a \textbf{mirroring strategy}.

\begin{thm} \label{thm1}
    Player N will win the Sign Game played on $K_n$ for $n\geq 2$, unless Player N is Player 1 and either $n=2$, in which case Player P wins, or $n=4$, which results in a draw.
\end{thm}
\begin{proof}
    As in Lem. \ref{lem1}, suppose the Sign Game is played to completion on $G=K_n$ ($n \geq 2$) with $a$ vertices assigned $+1$ and $b$ vertices assigned $-1$. Without loss of generality, we can assume $a \geq b$, so that $a = b+r$ for some $r \geq 0$. By Lem. \ref{lem1}, we have

    \begin{align}
        s(G) &= \binom{a}{2} + \binom{b}{2} -ab \\
        &= \binom{b+r}{2} + \binom{b}{2}-(b+r)b \\
        &=\frac{r^2-r}{2}-b,
    \end{align}

    \noindent which is negative precisely when 

    \begin{equation} \label{completeinequality}
         \frac{r^2-r}{2}<b.
    \end{equation}

    Player N can guarantee $r \leq 2$ by mirroring any turn by Player P with an opposite vertex assignment. Should Player N have the first turn of the game, then any vertex assignment, followed by this strategy, will ensure that $r \leq 2$.

    Note that when $r=0$ or $r=1$, Eqn. \ref{completeinequality} is satisfied so long as $b \geq 1$. The described mirroring strategy yields such a value of $b$ (and consequently $r$) in the cases when Player P has the first turn, or, when $n$ is odd.

    For all other scenarios, if $r=2$, then $s(G) < 0$ when $1 < b$ (that is, Player N will win as long as there are at least two vertices assigned a $-1$). Following the described mirroring strategy, this will occur for all cases except $n =2$ or $n=4$ along with Player N having the first turn.

    So, we must consider these cases: playing on $K_2$ or $K_4$ with Player N being Player 1. The first, on $K_2$, is obvious as Player P can guarantee $s(G)=1$. When $G = K_4$, there are only two possible outcomes (as players are playing optimally) to consider. The first, $a = 3$ and $b =1$ (or equivalently $a = 1$ and $b = 3$) yields $s(G) = 0$. The other possibility of $a = b = 2$ results in a score of $-2$. Player P can guarantee this does not happen by mirroring Player N and assigning vertices identically, meaning either $a$ or $b$ is at least $3$. Thus, playing on $K_4$ results in a draw.
\end{proof}

\subsection{Star graphs}

Trees are common graphs that can represent hierarchical relationships. One particular type of tree is a \textbf{star graph}. In an office setting, a single boss may oversee multiple departments, but those departments have no relationships between themselves. This could be represented by a star graph, denoted by $S_n$, a graph with $n+1$ vertices, one of which is of degree $n$ (called the \textbf{central vertex}) and the remaining $n$ vertices (called \textbf{leaves}) each of degree one. Figure \ref{fig2} shows $S_4$ and $S_2$.

Theorem \ref{thm2} proves that the Sign Game on star graphs is completely determined. The main tool in its proof is illustrated in Fig. \ref{fig2}, reducing a game on $S_k$ to a game on $S_{k-2}$ via SG-equivalence when two leaves on $S_k$ have been assigned opposite values.

\begin{thm} \label{thm2}
When played on $S_n$, the Sign Game results in a draw when $n$ is even and Player 2 wins when $n$ is odd.
\end{thm}
\begin{proof}
 For $n=2k$, one of the players will win if and only if at least $k+1$ of the edges have the same score. Because of this, neither player will mirror the assignment of a leaf with the same value that was just assigned. Thus, every two subsequent leaf assignments will reduce via SG-equivalence to a star graph with two less leaves. Hence, we need only consider the game being played on $S_2$, which clearly results in a draw.

When played on $S_n$ with $n = 2k+1$, similar repeated SG-equivalence reduces the game to being played on $S_1$, in which case the second player wins.
\end{proof}

When the Sign Game is played on a single star graph, the players have little choice about where they play (if the central vertex has been assigned a value, then the choice of leaf a player assigns has no impact on the final score of the game). But if the game were to be played on a disjoint collection of star graphs, decisions must be made about which component to play on. 

For graphs whose disjoint components consist specifically of star graphs, just as in Thm. \ref{thm2}, the parity of the leaves of the graphs determines the outcome of the game. The proof of the result, Thm. \ref{thmstars}, relying on mirroring and SG-equivalence similar to that in the proof of Thm. \ref{thm2}, is left to the reader (and it should be noted that Thm. \ref{thm2} is a special case of this result).

\begin{thm} \label{thmstars}
When the Sign Game is played on a disjoint union of star graphs, the game's result is as follows.

\begin{enumerate}
    \item The game is a draw if all components have an even number of leaves.
    \item Player 2 wins if at least one component has an odd number of leaves and the number of components with an even number of leaves is even.
    \item Player 1 wins if at least one component has an odd number of leaves and the number of components with an even number of leaves is odd.

\end{enumerate}
\end{thm}

\subsection{Bi- and tripartite graphs}

As mentioned in the introduction, graphs are oftentimes used to represent social networks. When the connections in this network occur between two different kinds of things, like users and events, the resulting graphs are called \textbf{bipartite}.  The \textbf{complete bipartite graph} $K_{m,n}$ ($m,n \geq 1$) is the graph consisting of $m+n$ vertices that can be partitioned into two nonempty sets $V_m$ and $V_n$, called the \textbf{partitioning sets} of $K_{m,n}$, where every vertex in $V_m$ is adjacent to each vertex in $V_n$ and no vertex in a partitioning set is adjacent to any other vertex in the same partitioning set. 

The following lemma is critical to this section's main result, Thm. \ref{bipartitetheorem}. Its proof is left to the reader.

\begin{lem} \label{bipartitelemma}
 When the Sign Game is played on $G = K_{m,n}$, with partitioning sets $V_m$ and $V_n$, and exactly $a$ vertices of $V_m$ and $b$ vertices of $V_n$ are assigned a value of $+1$, then $s(G) = (2a-m)(2b-n)$.
\end{lem}

Note that the star graph $S_n$ is a complete bipartite graph of the form $K_{1,n}$. And just as in Thm. \ref{thm2}, Player 2 will win or force a draw.

\begin{thm} \label{bipartitetheorem}
When the Sign Game is played on $K_{m,n}$, Player 2 wins when both $m$ and $n$ are odd and forces a draw when at least one of $m$ or $n$ is even.
\end{thm}
\begin{proof}
Suppose first that both $m$ and $n$ are odd. By mirroring Player 1 and playing in the same partitioning set as Player 1 just played in, but assigning a vertex the opposite value of what Player 1 just played, Player 2 will guarantee that, because $m$ and $n$ are both odd, that Player 1 will ultimately assign the last unassigned vertex in one of the partitioning sets. When this happens, Player 2 will assign a vertex in the other partitioning set. When Player 2 is Player P, they will assign this vertex the same sign Player N assigned on the previous turn. When Player N is Player 2, they assign the vertex the opposite sign as to what Player P just assigned.

Such a strategy results in an SG-equivalent complete bipartite graph $K_{1,n}$. Because $n$ is odd and Player 2 assigns $\lceil \frac{n}{2} \rceil$ vertices in the partitioning set of size $n$, Player 2 is guaranteed victory.

Now suppose that at least one of $m$ or $n$ is even. In the notation of Lem. \ref{bipartitelemma}, Player P wins if both $2a > m$ and $2b > n$, and, Player N wins if $2a > m$ and $2b < n$ (or vice versa). But note that when two vertices in the same partitioning set are assigned opposite values, the resulting graph is SG-equivalent to the graph with those two vertices and incident edges removed.

The other key observation is that Player P banks a point only when two vertices in opposite partitioning sets have the same value; Player N banks a point if they have opposite values. This means Player N will mirror Player P by always assigning a vertex in the opposite partitioning set the opposite assignment. 

This mirroring strategy means that Player P cannot win and will thus play for a draw. Player P can choose to play successive turns in only a partitioning set with an even number of vertices, guaranteeing exactly half these vertices are assigned $+1$ and half assigned $-1$ (dependent upon whether Player P or Player N plays first, and, who ultimately assigns the first vertex in the even-sized partitioning set). In this case, then, the resulting score is $0$ and the game a draw.

\end{proof}

The notion of bipartite graphs can be generalized to graphs whose sets are partitioned into $k$ partitioning sets. Such a graph is called a \textbf{$k$-partite graph} (when $k=3$, the graph is called \textbf{tripartite}), and when there is an edge between every pair of vertices from different partitioning sets, the graph is called a \textbf{complete $k$-partite graph}, denoted, in the case of a complete tripartite graph with partitioning sets of size $l$, $m$, and $n$, as $K_{l,m,n}$. The result of Thm. \ref{bipartitetheorem} is conjectured to generalize to the following.\\

\noindent \textbf{Conjecture} \textit{The result of the Sign Game being played on $K_{l,m,n}$ is as follows.}

\begin{enumerate}
    \item \textit{If all of $l$, $m$, and $n$ are odd, Player N wins.}
    \item \textit{If exactly one of $l$, $m$, or $n$ is even, Player 2 wins.}
    \item \textit{If exactly two or three of $l$, $m$, and $n$ are even, the result is a draw.}
\end{enumerate}

There is likely a generalization of Thm. \ref{bipartitetheorem} and the above conjecture to complete $k$-partite graphs, a question itself worth investigating.

\subsection{Path graphs}

Another type of tree graph is the type as in Fig. \ref{figpath1}. This is an example of a \textbf{path graph}. In general, the path graph $P_n$ is a graph with $n$ vertices that can be ordered $v_1$, $v_2$, $\ldots$, $v_n$ and $n-1$ edges $\left(v_i,v_{i+1}\right)$, for $i = 1,2,\ldots,n-1$. The \textbf{length} of $P_n$ is said to be $n$.

\begin{figure}[htb] 
\centering
\begin{tikzpicture} 
%\draw[step=1cm,gray,very thin] (0,0) grid (2,2);
\draw (0,0)--(1,0)--(2,0)--(3,0);
\draw[black,fill=black] (0,0) circle (.5ex);
\draw[black,fill=black] (1,0) circle (.5ex);
\draw[black,fill=black] (2,0) circle (.5ex);
\draw[black,fill=black] (3,0) circle (.5ex);

\end{tikzpicture}
\caption{The graph $P_4$}
\label{figpath1}
\end{figure}
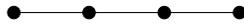

The result of playing the Sign Game on $P_n$ is similar to when it is played on $S_n$, with the winner dependent only upon the parity of $n$. 

\begin{thm} \label{thm3}
When played on $P_n$, the Sign Game is a draw when $n$ is odd and Player 2 wins (with a score of $1$ point) when $n$ is even.
\end{thm}

The proof of the result is inductive, and to employ the induction, we observe that the graph $P_n$ with a single degree $2$ vertex $v_i$ assigned a value is SG-equivalent to the disjoint union of $P_{i}$ (with $v_i$ of this graph inheriting the assignment from the original graph) and $P_{n-i+1}$ (with $v_1$ of this graph inheriting the assignment from the original graph). See Fig. \ref{figpathSG}. 

\begin{figure}[htb]  
\centering  

\subfigure[$G_1$]  
{  
\begin{tikzpicture}[]  

\draw[black,fill=black] (1,0) circle (.5ex);
\draw[black,fill=black] (2,0) circle (.5ex);
\draw (3,0) circle (.22cm);
\node (20) at (3,0) {$+$};
\draw[black,fill=black] (4,0) circle (.5ex);

\draw (1,0)--(2.77,0);
\draw (3.23,0)--(4,0);

\end{tikzpicture}  

}  
% The only difference is here, where I have commented out an empty line.
\hspace{50pt}
\subfigure[$G_2$]  
{  

\begin{tikzpicture} 

\draw[black,fill=black] (1,0) circle (.5ex);
\draw[black,fill=black] (2,0) circle (.5ex);
\draw (3,0) circle (.22cm);
\node (20) at (3,0) {$+$};

\draw (4,0) circle (.22cm);
\node (20) at (4,0) {$+$};
\draw[black,fill=black] (5,0) circle (.5ex);

\draw (1,0)--(2.77,0);
\draw (4.23,0)--(5,0);
%\draw[thick] (1.15,1.85)--(1.85,0.15);

\end{tikzpicture}  

}

\caption{SG-equivalent graphs: path and union of paths}
\label{figpathSG}
\end{figure}
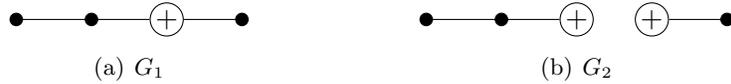

With this, we proceed in proving Thm. \ref{thm3}. 
\begin{proof}
    We proceed via induction. As $P_2 =S_1$ $P_3 = S_2$, the base cases follow from Thm. \ref{thm2}, we consider playing the Sign Game on $P_n$ for an arbitrary $n$. For $n$ even, any first vertex assignment (other than $v_1$ or $v_n$) will result in an SG-equivalent graph as in Fig. \ref{figpathSG}, where one of the components will have even length and the other odd. 

    In this case, Player 2 will choose to play on the component of even length, and then will subsequently mirror Player 1 by playing on the same component as Player 1, yielding, by induction, a score of zero being contributed from the component of odd length and a score in their favor on the component of even length.

    Being strategic, then, Player 1 will initially choose to play on $v_1$ or $v_n$, but Player 2 can choose to simply play on the adjacent vertex. The resulting two component graph consists of a path of length 2, banking a point for Player 2, and a path of odd length, which inductively will contribute zero points to the final score. Thus, in either case, when $n$ is even, Player 2 wins.

    When $n$ is odd and the first play is on a degree $2$ vertex, after the first play, the graph is SG-equivalent to a two component graph, each component a path graph, where the lengths of each component have matching parity. If both are of odd length, then each will contribute zero to the final score. If both are of even length, each player can be the second to play on one of the components. As each plays strategically, they will choose to then mirror each other so that each component will bank a single point for each player, resulting in a draw.

    Lastly, we need only consider when $n$ is odd and the first play is on a degree $1$ vertex. Without loss of generality, assume the first play is on $v_1$. We view the resulting graph, though only one component, as a nondisjoint union of two path components. The first, the path $P_2$ consisting of the two vertices $v_1$ and $v_2$, and the path $P_{n-1}$ containing the vertices $v_2$ through $v_n$. If Player 2 chooses to assign $v_2$, while they bank $1$ point from the $P_2$ component, they allow Player 1 to be the second player to play on the $P_{n-1}$ component, a path of even length, and consequently allow Player 1 to bank $1$ point. Should Player 2 assign any vertex other than $v_2$, Player 1 would subsequently assign $v_2$, banking $1$ point from each of the path components. Thus, Player 2's strategic move is to guarantee the draw.
\end{proof}

It is interesting to note the similarity in results for all the families of graphs investigated thus far, particularly from the perspective of Player P. The next family of graphs, cycles, exhibits a very different type of result, including cases where, for the first time, Player P is guaranteed victory. A few observations about playing the Sign Game on $P_5$ are needed for the forthcoming investigation into cycles. We state them in the lemma below, leaving its proof to the reader.

\begin{lem} \label{p5lemma}
    Suppose the Sign Game is played on $P_5$.

    \begin{enumerate}
    \item If the first two players label $v_1$ and $v_5$ on their first moves with opposite signs, then Player 1, who labels two of the remaining three vertices, can guarantee a victory. If the first two players label $v_1$ and $v_5$ on their first moves with the same sign, the Player 2 can force a draw.

    \item Up to symmetry and duality (that is, interchanging all $+1$ and $-1$ values), there are three possible assignments to the vertices of $P_5$ that result in a positive game score (illustrated as $+++++$, $++++-$, and $+++--$), three that result in a negative score ($++-+-$, $+-++-$, and $+-+-+$), and all others result in a score of $0$.
\end{enumerate}
\end{lem}

\subsection{Cycles}

A \textbf{cycle graph} $C_n$ is a graph with $n$ vertices ($n \geq 3$) whose vertices can be ordered $v_1$, $v_2$, $\ldots$, $v_n$ with edges $\left(v_i,v_{i+1}\right)$, for $i = 1,2,\ldots,n-1$ and $\left(v_n,v_1\right)$. Visually, $C_n$ is often drawn as an $n$-gon, with vertices at every corner. Of the graphs investigated in this paper, one subset of these are the first in which Player P has a winning strategy.

Before proceeding with the complete classification of the Sign Game on $C_n$, we introduce some useful terminology. Once the first vertex (for example, $v_1$) of $C_n$ has been assigned a value, the graph is SG-equivalent to $P_{n+1}$ with $v_1$ and $v_{n+1}$ inheriting the same value as was assigned $v_1$ on $C_n$. See Fig. \ref{cycleSG}.

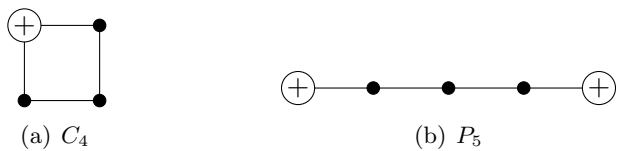
\begin{figure}[htb]  
\centering  

\subfigure[$C_4$]  
{  
\begin{tikzpicture}[]  

\draw[black,fill=black] (1,0) circle (.5ex);
\draw[black,fill=black] (2,0) circle (.5ex);
\draw[black,fill=black] (2,1) circle (.5ex);

\draw (1,1) circle (.22cm);
\node (20) at (1,1) {$+$};

\draw (1,.78)--(1,0)--(2,0)--(2,1)--(1.22,1);

\end{tikzpicture}  

}  
% The only difference is here, where I have commented out an empty line.
\hspace{50pt}
\subfigure[$P_5$]  
{  

\begin{tikzpicture} 

\draw[black,fill=black] (2,0) circle (.5ex);
\draw[black,fill=black] (3,0) circle (.5ex);
\draw[black,fill=black] (4,0) circle (.5ex);
\draw (1,0) circle (.22cm);
\node (20) at (1,0) {$+$};
\draw (5,0) circle (.22cm);
\node (20) at (5,0) {$+$};

\draw (1.22,0)--(4.78,0);

\end{tikzpicture}  

}

\caption{SG-equivalent graphs: $C_4$ and $P_5$}
\label{cycleSG}
\end{figure}

For purposes we will see in the proof of Thm. \ref{cycletheorem}, we proceed to partition the set of edges of this path graph $P_{n+1}$ into partitioning sets of size four (referred to as the \textbf{$4$-segments} of $P_{n+1}$) and, if $n$ mod $4 = k > 0$, a single partitioning set of size $k$ (called the \textbf{$k$-segment} of $P_{n+1}$). The first $4$-segment contains edges $(v_1,v_2)$ through $(v_4,v_5)$, the second contains $(v_5,v_6)$ through $(v_8,v_9)$, etc. The $k$-segment contains edges $(v_{n+1-k},v_{n+2-k})$ through $(v_{n},v_{n+1})$. For each segment, whether it is a $4$-segment or the $k$-segment, we call the first and last vertex of the segment the \textbf{endvertices} of the segment, while the other vertices will be referred to as \textbf{interior} vertices of the segment.

This idea is illustrated in Fig. \ref{4segments}. If $v_1$ of $C_{11}$ has been assigned a $+1$, then this partiually labeled graph is SG-equivalent to $P_{12}$ where $v_1$ and $v_{12}$ inherit the $+1$ value. In Fig. \ref{4segments}, edges in the first $4$-segment are represented via solid lines, while those in the second $4$-segment are represented via dotted lines, and those in the $3$-segment are illustrated with dashed lines. Vertices $v_5$ and $v_9$, two of the endvertices, are larger to help visualize the segments.\\

\begin{figure}[htb]
\centering
\begin{tikzpicture}

\draw (1,0) circle (.22cm);
\node (00) at (1,0) {$+$};
\draw[black,fill=black] (2,0) circle (.5ex);
\draw[black,fill=black] (3,0) circle (.5ex);
\draw[black,fill=black] (4,0) circle (.5ex);
\draw[black,fill=black] (5,0) circle (1.2ex);
\draw[black,fill=black] (6,0) circle (.5ex);
\draw[black,fill=black] (7,0) circle (.5ex);
\draw[black,fill=black] (8,0) circle (.5ex);
\draw[black,fill=black] (9,0) circle (1.2ex);
\draw[black,fill=black] (10,0) circle (.5ex);
\draw[black,fill=black] (11,0) circle (.5ex);
\draw (12,0) circle (.22cm);
\node (00) at (12,0) {$+$};

\node (00) at (1,-.5) {$v_1$};
\node (00) at (5,-.5) {$v_5$};
\node (00) at (9,-.5) {$v_9$};
\node (00) at (12,-.5) {$v_{12}$};
\draw [solid] (1.23,0) -- (5,0);
\draw [dotted] (5,0) -- (9,0);
\draw [dashed] (9,0) -- (11.77,0);

\end{tikzpicture}
\caption{Two $4$-segments and one $3$-segment}
 \label{4segments}
\end{figure}
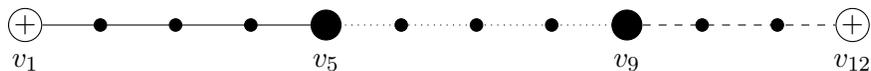

\begin{thm} \label{cycletheorem}
    If $n$ mod $4 =  k$, then the result of the Sign Game on $C_n$ is a draw when $k = 0$, Player P wins when $k = 1$, Player 2 wins when $k = 2$, and Player N wins when $k = 3$.
\end{thm}
\begin{proof}
    We consider the four cases ($k=0$, $1$, $2$, and $3$), beginning with an observation of what can happen when played on $C_3$ through $C_6$.

     First, note that Player N will always win the Sign Game played on $C_3$ since they can guarantee at least two vertices will have different assigned values. The only vertex assignments on $C_4$ that result in a nonzero score are when all vertices have the same assignment (resulting in a score of $4$), or when $v_1$ and $v_3$ are assigned one value and $v_2$ and $v_4$ are assigned the other (resulting in a score of $-4$). Both players playing optimally will prevent this from happening.

     Of the final vertex assignments on $C_5$, the only ones resulting in a negative score are when, without loss of generality, $v_1$ and $v_3$ have the same assigned value and the remaining three vertices all have the other assigned value. See Fig. \ref{c5base}. Because Player 1 will assign three of the vertices in $G$, when this is Player P, they can guarantee such an arrangement does not occur (and consequently that $s(G) > 0$ since there are an odd number of edges). When Player P plays second, by utilizing a mirroring strategy and playing on a vertex adjacent to what Player N just assigned, and assigning the same value, Player P can guarantee the game results in a positive score.

\begin{figure}[htb] 
\centering
\begin{tikzpicture}
\draw[thick] (.84,.18)--(0.15,1.32);
\draw[thick] (.15,1.68)--(1.85,2.85);
\draw[thick] (3.85,1.68)--(2.15,2.85);
\draw[thick] (3.16,.18)--(3.85,1.32);
\draw[thick] (1.22,0)--(2.78,0);

\draw (1,0) circle (.22cm);
\node (00) at (1,0) {$+$};
\draw (3,0) circle (.22cm);
\node (00) at (3,0) {$+$};
\draw (0,1.5) circle (.22cm);
\node (00) at (0,1.5) {$-$};
\draw (4,1.5) circle (.22cm);
\node (00) at (4,1.5) {$-$};
\draw (2,3) circle (.22cm);
\node (00) at (2,3) {$+$};

\end{tikzpicture}
\caption{$s(G) < 0$}
\label{c5base}
\end{figure}

The proof of $C_6$ is similar to that of $C_5$ and is left to the reader.

We proceed now to consider arbitrary length cycle graphs. When $k = 0$, after the first turn, we can consider playing the game on $P_{n+1}$, where $v_1$ and $v_{n+1}$ inherit the initial vertex assignment. Consider just the sequence of endvertices: $v_1$, $v_5$, $v_9$, $\ldots$, $v_{n+1}$. When all these vertices have been assigned values, since $v_1$ and $v_{n+1}$ have the same assigned value, the number of changes in assignment value (e.g., $v_1$ is assigned $+1$ and $v_5$ assigned $-1$ would be an example of a change in assignment value) must be even. Thus, the number of $4$-segments whose endvertices have opposite assigned values is even. Lemma \ref{p5lemma} tells us that the player to assign at least two of the non-endvertices of one of these $4$-segments will bank at least $2$ points (and will bank $4$ points if they assign all three non-endvertices). 

Thus, when one player becomes the first to play a non-endvertex on one of these $4$-segments, the other player will choose to mirror and be the first to play a non-endvertex of another of these $4$-segments (and because there is an even number of them, this will always be possible). Optimal play also guarantees that on any $4$-segment with two endvertices assigned the same value, each player will assign at least one of the interior vertices, and Lem. \ref{p5lemma} gives that this $4$-segment contributes nothing to the game's final score. Ultimately, when $k =0$, the score of the Sign Game is $0$, resulting in a draw.

Next, assume $k = 1$ and $n > 5$ (as the $n = 5$ case was specifically proven above). When Player N plays first, Player P can assign $v_{n-4}$ the same value as what Player N assigned $v_{n+1}$. The resulting graph is SG-equivalent to a disjoint union of path graphs, one of which is $P_6$ with only its first and last vertices assigned the same value, and, the other whose length is a multiple of $4$, with only its first and last vertices assigned the same value. By Player P mirroring, the $P_6$ component will contribute positively to the game's final score, as per the discussion of $C_5$ above. The other component, via an identical argument to the $k=0$ case, contributes zero to the final score.

In the case that $k = 1$ and Player P plays first, regardless of what vertex Player N assigns on their first move, Player P can assign either $v_6$ or $v_{n-4}$ the same assigned value as $v_1$ (and $v_{n+1}$), so that the resulting graph is SG-equivalent to a disjoint union of path graphs, one of which is length $6$ with only its first and last vertices assigned the same value, and, the other whose length is a multiple of $4$, with only one degree $2$ vertex assigned a value (that Player N first assigned). As when Player N played first, Player P can simply employ a mirroring strategy to guarantee victory, as the component whose length is a multiple of $4$ only has a single degree $2$ vertex assigned a value (so that the $k=0$ argument still holds).

When $k = 2$, Player 2 will choose to assign $v_{n-5}$ the same value as Player 1 assigned $v_1$ (and consequently $v_{n+1}$). The resulting graph is SG-equivalent to the disjoint union of a $P_7$ graph with its first and last vertices assigned the same value, and, a path graph whose length is a multiple of $4$, also with its first and last vertices assigned the same value. The first of these is SG-equivalent to $C_6$, which Player 2, via a mirroring strategy, can guarantee banks points in their favor. The latter, as in the $k=0$ case, will contribute no points to the game's final score. Hence, Player 2 wins the game in this case.

When $k = 3$, we consider separately the scenarios determined by which player has the first move. In either case, we will consider playing the game on $P_{n+1}$, viewed as $4$-segments and a single $3$-segment with $v_1$ and $v_{n+1}$ inheriting the sign that was first played on $C_n$.

After Player P plays first, Player N can choose to assign $v_{n-2}$ the same value as $v_{n+1}$. Then, by employing a mirroring strategy, Player N can guarantee the $3$-segment contributes negatively to the final score. The collection of $4$-segments, the path from $v_1$ through $v_{n-2}$, is SG-equivalent to the $k = 0$ case, contributing $0$ points to the final score.  Thus, Player N wins.

Consider next if Player N plays first. Regardless of what vertex Player P assigns on their first turn, Player N can assign either $v_4$ or $v_{n-2}$ the same value as their first move, so that the resulting graph can be viewed as a single $3$-segment (with only its endvertices assigned matching values) and the remaining vertices grouped into $4$-segments.

Using a similar mirroring strategy as when Player P played first, Player N can guarantee the $3$-segment contributes negatively to the game's final score. The conclusion that the $4$-segments do not contribute to the game's final score is similar to that in the $k = 1$ case.

Ultimately, we have that Player N wins the Sign Game when $k=3$, consequently proving the result.

\end{proof}

\section{Conclusion}

The simplicity of this game lends itself to a plethora of interesting questions worthy of exploration. Combinations of these possibilities and/or incorporating deeper mathematical techniques opens the door for even more considerations. A perusal of Gallian's growing survey of graph labeling \cite{2} would likely yield even more approaches.

In the Sign Game in this paper, the two players had free choice as to what vertex they would assign and what assignment they would give it. In the context of the introduction (political support in a social network), it is likely that individuals feel some sort of pressure to conform from their connections (pressure to assign a certain value to a vertex) while political groups themselves may prioritize influencing certain individuals (which vertices should be assigned on a given turn). Could a weighting system or probability function be incorporated into the game to mimic these scenarios?

When it comes to social networks, certain individuals (influencers) have a lot of sway. In this game, this could involve a probability function assigned to one of the players so that they may get an immediate additional turn following one of their turns. Perhaps risk becomes involved; if the player chooses to ``take a chance" to get an additional turn but does obtain it, and in doing so the other player will have back-to-back turns, when is it advantageous to take such a gamble?

In politics, there are also those people who just do not care. How does the game play if vertices can be assigned a value of $0$? What if the game is expanded to three players and the goal of the third player is to end the game with a score of $0$?

The variants listed above are all related to vertex assignments, but another clear way to alter the game is to alter the graphs. Like any game played on mathematical graphs, changing the type of graph can result in a brand new and very different game. Studying the Sign Game on families of graphs beyond those investigated here could lead to further interesting results, as could adjusting the rules and extending the game to different types of graphs including multigraphs, directed graphs, or graphs with weighted edges.

\end{document}